\documentclass[A4paper,12pt]{article}
\usepackage{latexsym}
\usepackage{mathrsfs}
\usepackage{amssymb}
\usepackage{amscd}
\usepackage[dvips]{graphicx}                  
\newtheorem{definition}{Definition}
\newtheorem{theorem}{Theorem}

\newtheorem{lemma}{Lemma}

\newenvironment{proof}[1][Proof]{\textbf{#1.} }{\ \rule{0.5em}{0.5em}}
\date{}
\long\def\symbolfootnote[#1]#2{\begingroup%
\def\thefootnote{$\;$}\footnote[#1]{$^*$#2}\endgroup}
\begin{document}

\title{On Kuratowski partitions in the Marczewski structure and Ellentuck topology}
\author{Ryszard Frankiewicz and Joanna Jureczko}

\maketitle

\symbolfootnote[2]{Mathematics Subject Classification: 54B15 (Primary), 54D20 (Secondary). 
	
\hspace{0.2cm}
}
\begin{abstract}
We show that large sets in the Marczewski structure and Ellentuck topology do not admit Kuratowski partitions. 
\end{abstract}

\section{Introduction}

In 1935  K. Kuratowski  posed the problem (\cite{KK1}) which is equivalent to non-existence of Kuratowski partitions in the term of functions. It was found to be equivalent to the following result, (see \cite{KK2} and \cite{EFK1}). 
\\ 
\\
\textbf{Theorem (Kuratowski)}
\textit{(CH) If $\mathcal{F}$  is a partition of the interval $[0,1]$ of cardinality at most $\omega_1$ into meager sets then there exists a family $\mathcal{A} \subseteq \mathcal{F}$ such that $\bigcup \mathcal{A}$ has not the Baire property.}
\\

Later R. Solovay in \cite{RS} and L. Bukovsky in \cite{LB} proved this result without (CH) assumption using the forcing method and natural embedding in the proof.
In \cite{EFK} the authors gave a generalization of this result for a class of spaces of weight $\leq 2^\omega$. 
In the light of Kuratowski theorem we accept the following definition, (compare \cite{FS}).

\begin{definition}
Let $X$ be a Baire topological space. 
A partition $\mathcal{F}$ of $X$  is a Kuratowski partition if  for each subfamily $\mathcal{F}' \subseteq \mathcal{F}$ the set $\bigcup \mathcal{F'}$ has the Baire property. 
\end{definition}

During considerations on this topic the natural problem has been arisen whether this theorem is also true in other spaces as $(s)$-structure (Sacks forcing) and completely Ramsey sets in  Ellentuck topology (Mathias forcing).

From a technical point of view and proofs of theorems, investigations of $(s)$-structure and completely Ramsey sets are similar because in both cases it is impossible to construct  Bernstein-type sets for large sets directly and indicate the cardinality of large partitions in both spaces. Such construction requires using Fusion Lemma, (see \cite{TJ, ST}), which was introduced in perfect sets forcing, (i.e. Sacks forcing).

In this paper we show that without (CH) there are no Kuratowski partitions in both considered spaces just using Fusion Lemma in the proofs.
Moreover, we show that the existence of large cardinals has not an influence on the $(s)$-structures and the Ellentuck structure. We provide considerations for $(s)$-sets before considerations for Ellentuck  sets, because the $(s)$-structure is more readable than the Ellentuck structure, which seems to be in some sense "polydimensional".

The last theorem of this paper was published for the first time in \cite{FS}, but the proof presented there was incorrect, which was discovered after publication. The mistake relied on eliminating the fusion during counting the cardinality of a partition of large Ellentuck sets restricted to perfect sets. The mistake is repaired here which implies the change of method using in the proof. The first author apologizes for the mistake and informs that although theorem remains true, the arguments using in the proof presented in \cite{FS} cannot be repaired.

However, reasonings presented below are restricted to $(s)$-sets and Ellentuck sets, also remain true for a wider class of sets.

It is mostly provable that the further results can be obtained by investigations of quotient algebra structure, i. e. $P(\kappa)/I$, where $I$ is a $\kappa$-complete ideal on $\kappa$. 

In the paper we use standard notation and terminology for the field. For definitions and facts not cited here we refer the reader to \cite{TJ, KK, ST}.

\section{Partitions in the Marczewski structure}

Let $\mathbb{R}$ be a real line with natural topology.
A set $A \subseteq \mathbb{R}$ is an \textit{$(s)$-set} if for each perfect set $P \subseteq \mathbb{R}$ there exists a perfect set $P'\subseteq P$ such that $P' \subseteq A \cap P$ or $P' \subseteq P \setminus A$.
A set $A \subseteq \mathbb{R}$ is \textit{an $(s_0)$-set} if for each perfect set $P \subseteq \mathbb{R}$  there exists a perfect set $P'\subseteq P$ such that $P' \cap A = \emptyset$.
A set $A\subseteq \mathbb{R}$ is \textit{large} if it is not an $(s_0)$-set.
A family $\mathcal{F}$ of subsets of $A$ is \textit{large} if $\bigcup \mathcal{F}$ is large.

Let $A \subseteq \mathbb{R}$ be a large set and let $\mathcal{F}$ be a large partition of $A$ into $(s_0)-sets.$  The set $A$ is \textit{dense in $(s)-sense\  ($s$-dense)$} if for any perfect set $P\subseteq \mathbb{R}$ such that $A\cap P$ is large the family $\{F \in \mathcal{F} \colon F\cap P\not = \emptyset \}$ is also large.
\begin{lemma}
Let $A \subseteq \mathbb{R}$  be a large and $(s)-$dense set. For any large partition $\mathcal{F}$ of $A$ into $(s_0)$-sets and for any perfect set $P\subseteq R$ the family
$$\mathcal{F}_P = \{F \cap P \colon F_\alpha \in \mathcal{F}\}$$ has cardinality continuum.
\end{lemma}

\begin{proof}
Divide $\mathcal{F}$ into disjoint large families $\mathcal{F}_0, \mathcal{F}_1$ such that there are disjoint perfect sets $P_0, P_1 \subseteq R$ such that $A\cap P_0$ and $A\cap P_1$ are large and $A\cap P_0 \subseteq \bigcup \mathcal{F}_0$ and $A\cap P_1 \subseteq \bigcup \mathcal{F}_1$.

Assume that for some $n \in \omega$ we have constructed a collection of families $\{\mathcal{F}_h \colon h \in {^n2} \}$ such that 
\\
(1) $\mathcal{F}_{h} \subseteq \mathcal{F}_{h|(n-1)} $,
\\
(2) $\mathcal{F}_{h} \cap \mathcal{F}_{h'} = \emptyset$ for any $h, h' \in {^n}2, h \not = h'$,
\\
(3) $\bigcup\{\mathcal{F}_{h} \colon h \in {^n2} \} = \bigcup \mathcal{F}$,
\\
(4) there are disjoint perfect sets $P_h \subseteq P_{h|(n-1)}$ such that $A\cap P_h$ is large and $A\cap P_h \subseteq \bigcup \mathcal{F}_h$.

Now fix $h \in {^n2}$ and divide $\mathcal{F}_h$ into disjoint large families $\mathcal{F}_{h^{\smallfrown} 0}, \mathcal{F}_{h^{\smallfrown} 1}$ such that there are disjoint perfect sets $P_{h^{\smallfrown} 0}, P_{h^{\smallfrown} 1} \subseteq P_h$ such that $A\cap P_{h^{\smallfrown} 0}$ and $A\cap P_{h^{\smallfrown} 1}$ are large and $A\cap P_{h^{\smallfrown} 0} \subseteq \bigcup \mathcal{F}_{h^{\smallfrown} 0}$ and $A\cap P_{h^{\smallfrown} 1} \subseteq \bigcup \mathcal{F}_{h^{\smallfrown} 1}$.

Continuing the construction for any $n \in \omega$ we obtain the collection of families $\{\mathcal{F}_f \colon f \in {^\omega 2}\}$ of the properties:
\\
(1') $\mathcal{F}_{f|n} \subseteq \mathcal{F}_{f'|m} $ for any $f, f' \in {^\omega}{2} $ and $m < n$,
\\
(2') $\mathcal{F}_{f|n} \cap \mathcal{F}_{f'|n} = \emptyset$ for any $f, f' \in {^\omega}2, f \not = f'$, 
\\
(3') $\bigcup\{\mathcal{F}_{f|n} \colon f \in {^\omega}2 \} = \bigcup \mathcal{F}$
\\
(4') there are disjoint perfect sets $P_{f|n} \subseteq P_{f|(n-1)}$ such that $A\cap P_{f|n}$ is large and $A\cap P_{f|n} \subseteq \bigcup \mathcal{F}_{f|n}$.

Now take $Q = \{P_f \colon P_f \cap A \textrm{ is large}, f \in {^\omega 2}\}$
and let  $$P = \{\bigcap P_{f|n} \colon P_f \in Q, f \in {^\omega 2}, n \in \omega\}.$$
By Fusion Lemma, (see e. g. Lemma 26.2 in \cite{ST}),  $P$ is perfect. Since $A$ is large and $(s)-$dense and $A\cap P$ is large, $\mathcal{F}_P = \{F \cap P \colon F \in \mathcal{F}\}$ has cardinality continuum.
\end{proof}

\begin{lemma}
Let $A \subseteq \mathbb{R}$ be a large and $(s)-$dense.
Let $\mathcal{P}$ be a family of all perfect sets $P\subseteq \mathbb{R}$ such that $A \cap P \not = \emptyset$. Then for each large partition $\mathcal{F}$ of $A$ of cardinality continuum into $(s_0)$-sets there exists a subfamily $\mathcal{F'} \subseteq \mathcal{F}$ of cardinlaity continuum with the property
$$\forall_{F \in \mathcal{F'}}\ \exists_{P \in \mathcal{P}} F \cap P \not = \emptyset$$
for which 
$\bigcup \mathcal{F'}$ is not an $(s)$-set.
\end{lemma}

\begin{proof}
Enumerate 
$\mathcal{P} = \{P_\alpha \subseteq \mathbb{R} \colon P_\alpha \textrm{ is perfect and } A\cap P_\alpha \not = \emptyset, \alpha < 2^\omega\}$.

Let $\mathcal{F}$ be a large partition of $A$ of cardinality continuum into $(s_0)$-sets. By Lemma 1 there are perfect sets $P\subseteq \mathbb{R}$ such that $\mathcal{F}_P = \{F \cap P \colon F \in \mathcal{F}\}$ has cardinality continuum.
Hence for each $P_\alpha \in \mathcal{P}$ we can choose disjoint sets
$$B^{0}_{\alpha}, B^{1}_{\alpha} \in \{F \in \mathcal{F} \colon F \cap P_\alpha \not = \emptyset\} \setminus (\{B^{0}_{\beta} \colon \beta < \alpha\}\cup \{B^{1}_{\beta} \colon \beta < \alpha\}).$$
Now for each $\varepsilon \in \{0, 1\}$ consider families 
$\mathcal{B}^{\varepsilon} = \{B^{\varepsilon}_{\alpha} \colon \alpha < 2^\omega\}$. Obviously by the construction 
$\mathcal{B}^{0}\cap \mathcal{B}^{1} = \emptyset.$ 

The sets $\bigcup \mathcal{B}^{\varepsilon}$ are not $(s)$-sets.
Indeed. Suppose that  $\bigcup \mathcal{B}^{\varepsilon}$ is an $(s)$-set for some $\varepsilon \in \{0, 1\}$. Then there exists $P_\alpha \in \mathcal{P}$ such that $P_\alpha \cap \bigcup \mathcal{B}^{\varepsilon} = \emptyset$. But by the construction we have that 
$\{F \in \mathcal{F} \colon F \cap P_\alpha \cap \bigcup \mathcal{B}^{\varepsilon} \not = \emptyset\}$ is non-empty. A contradiction.

If $\bigcup \mathcal{B}^{\varepsilon}$ is not an $(s)$-set for some $\varepsilon \in \{0, 1\}$, then there exists $P_\alpha \in \mathcal{P}$ such that $P_\alpha \subseteq \bigcup \mathcal{B}^{\varepsilon}$ and by the construction $P_\alpha \cap \bigcup \mathcal{B}^{1 - \varepsilon} \not = \emptyset$ which contradicts with disjointness of families $\mathcal{B}^{0}$ and  $\mathcal{B}^{1}$. 
\end{proof}

\begin{theorem}
Let $A \subseteq \mathbb{R}$ be a large and $(s)-$dense set and let $\mathcal{F}$ be a large partition of $A$ of cardinality continuum into $(s_0)-$sets. Then $\mathcal{F}$ is not a Kuratowski partition.
\end{theorem}

\begin{proof}
Let $A \subseteq \mathbb{R}$ be a large and $(s)-$dense set and let $\mathcal{F}$ be a large partition of $A$ of cardinality continuum into $(s_0)-$sets.
Suppose in contrary that $\mathcal{F}$ is a Kuratowski partition. By Lemma 1 there are perfect sets $P \subseteq \mathbb{R}$ such that 
$\mathcal{F}_P = \{F \cap P \colon F \in \mathcal{F}\}$  has cardinality continuum. By Lemma 2 there exists a large subfamily $\mathcal{F'} \subseteq \mathcal{F}$ such that $\bigcup \mathcal{F'}$ is not an $(s)$-set. A contradiction.
\end{proof}

\section{Partitions in the Ellentuck structure}

The Ellentuck topology $[\omega]^{\omega}_{EL}$ on $[\omega]^\omega$ is generated by sets of the form
$$[a, A] = \{B \in [A]^\omega \colon a \subset B \subseteq a \cup A\},$$
where $a \in [\omega]^{<\omega}$ and $A \in [\omega]^\omega$. We call such sets \textit{Ellentuck sets, (or $EL$-sets).} Obviously $[a, A] \subseteq [b, B]$ iff $b \subseteq a$ and $A \subseteq B$. 

A set $M \subset [\omega]^\omega$ is \textit{completely Ramsey}, (or $M$ is a \textit{$CR$-set}), if for every $[a, A]$ there exists $B \in [A]^\omega$ such that $[a, B] \subset M$ or $[a, B] \cap M = \emptyset.$ 
A set $M \subset [\omega]^\omega$ is \textit{nowhere Ramsey}, (or $M$ is an \textit{$NR$-set}), if for every $[a, A]$ there exists $B \in [A]^\omega$ such that $[a, B] \cap M = \emptyset.$
An Ellentuck set is \textit{large} if is not an $NR$-set.

Let $M \subseteq [\omega]^\omega$ be a large and $EL$-open set and $\mathcal{F}$ be a large partition of $M$ into $NR$-sets. $M$ is \textit{dense in Ellentuck sense ($EL$-dense)} if for each $[a, A] \subseteq [\omega]^{\omega}_{EL}$ such that $M \cap [a, A]$ is large the family $\{F \in \mathcal{F} \colon F \cap [a, A] \not = \emptyset\}$.
\\
\\
\noindent
\textbf{Fact 1 (\cite{P})} \textit{Let $M$ be a $EL-$dense and $EL$-open set. Then for each $A \subseteq [\omega]^\omega$ there exists $B \subseteq [\omega]^\omega$ such that $B \subseteq A$ and for each $a \in [\omega]^{\omega}$ the set $[\emptyset, B \cup a] \subseteq M$.}    

\begin{lemma}
Let $M \subseteq [\omega]^\omega$ be a large and $EL$-dense $EL$-open set. For any large partition $\mathcal{F}$ of $M$ into $NR$-sets and for any $[a, A] \subseteq [\omega]^{\omega}_{EL} $ the family $\mathcal{F}_{[a, A]} = \{F \cap [a, A] \colon F \in \mathcal{F}\}$ has cardinality continuum.
\end{lemma}

\begin{proof}
We will construct a collection of families $\{\mathcal{F}_f \colon f \in {^\omega 2} \}$ of $\mathcal{F}$ of the following properties, for any $n, m \in \omega$:
\\
(1) $\mathcal{F}_{f|n} \subseteq \mathcal{F}_{f'|m} $ for any $f, f' \in {^\omega}{2} $ and $m < n$,
\\
(2) $\mathcal{F}_{f|n} \cap \mathcal{F}_{f'|n} = \emptyset$ for any $f, f' \in {^\omega}2, f \not = f'$, 
\\
(3) $\bigcup\{\mathcal{F}_{f|n} \colon f \in {^\omega}2 \} = \bigcup \mathcal{F}$,
\\
(4) there are disjoint  sets $[a_{f|n}, A_{f|n}] \subseteq [a_{f|(n-1)}, A_{f|(n-1)}]$ such that $M\cap [a_{f|n}, A_{f|n}]$ is large and $A\cap [a_{f|n}, A_{f|n}] \subseteq \bigcup \mathcal{F}_{f|n}$.

Using Fact 1 we can start with some $[\emptyset, B]$ where $B \subseteq A$.

Assume that for some $n \in \omega$ we have constructed families $\{\mathcal{F}_h \colon h \in {^n 2}\}$ of properties (1) - (4).

Fix $h \in {^n 2}$. Now we will divide $\mathcal{F}_h$ into disjoint large  subfamilies $\mathcal{F}_{h^{\smallfrown} 0}$ and $\mathcal{F}_{h^{\smallfrown} 1}$ in the followin way, (we use construction similar to Mathias forcing~\cite{JB}).

Take $[a_h, A_h]$ associated with $\mathcal{F}_h$, where $a_h = \{b_1, b_2, ..., b_n\}$. Enumerate all subsets of $a_h$ by $s_1, ..., s_k,$ where $ k = 2^n$. Construct the sequences of subsets of $A_h$:
$$C^{h}_{0} \supseteq C^{h}_{1}\supseteq .. \supseteq C^{h}_{k}
\textrm{ and }
D^{h}_{0} \supseteq D^{h}_{1}\supseteq .. \supseteq D^{h}_{k}$$ 
as follows: let $C^{h}_{0},  D^{h}_{0} \subseteq [A_{h}\setminus a_h]^\omega $ and $C^{h}_{0}\cap D^{h}_{0}  = \emptyset$.
Given $C^{h}_{i}, D^{h}_{i}$ if there exist $C\subseteq C^{h}_{i}$ and $D\subseteq D^{h}_{i}$ such that
$M\cap [s_i, C] \subseteq  \bigcup \mathcal{F}_{h\smallfrown 0}$
and 
$M\cap [s_i, D] \subseteq  \bigcup \mathcal{F}_{h\smallfrown 1}$
and $\mathcal{F}_{h\smallfrown 0}, \mathcal{F}_{h\smallfrown 1}$ are large and disjoint then $C^{h}_{i+1}:=C$ and $D^{h}_{i+1}:=D$. If not then then we take $C^{h}_{i+1}:=C^{h}_{i}$ and $D^{h}_{i+1}:=D^{h}_{i}$.

Let $b_{n+1}: = \min C_{k}^{h}$. Take
$$[a_{h^\smallfrown 0}, A_{h^\smallfrown 0}] = [a_h \cup b_{n+1}\}, C_{k}^{h} \setminus \{b_{n+1}\}]$$ and 
$$[a_{h^\smallfrown 1}, A_{h^\smallfrown 1}] = [a_h \cup \{b_{n+1}\}, C_{k}^{h} \setminus \{b_{n+1}\}].$$
Thus we have constructed the collection of families $\{\mathcal{F}_f \colon f \in {^\omega 2}\}$ and associated with then $[a_f, A_f] \subseteq [\omega]^{\omega}_{EL}$ fulfilling (1) - (4).

Using Fusion Lemma, (see e.g. \cite{TJ}  or \cite{ST}) to 
$$[\emptyset, B] = \bigcap \bigcup \{[a_{f|n}, A_{f|n}] \colon f \in {^\omega}2, n \in \omega\}$$
we obtain that $\mathcal{F}_{[\emptyset, B]}$ has cardinality continuum.
\end{proof}

\begin{lemma}
Let $M \subseteq [\omega]^\omega$ be a large and $EL-$dense $EL$-open set. Let $\mathcal{P}$ be a family of all perfect sets $P\subseteq [\omega]^\omega$ such that $M \cap P \not = \emptyset$.
Then for each large partition $\mathcal{F}$ of $M$ of cardinality continuum into $NR$-sets there exists a subfamily $\mathcal{F}' \subseteq \mathcal{F}$ with the property 
$$\forall_{F \in \mathcal{F}'}\ \exists_{P \in \mathcal{P}}\ F \cap P \not = \emptyset$$
for which $\bigcup \mathcal{F}'$ is not a $CR$-set.
\end{lemma}

\begin{proof}
Let $$\mathcal{P} = \{[a_\alpha, A_\alpha]\subseteq [\omega]^\omega \colon [a_\alpha, A_\alpha] \cap M \textrm{ is a non-empty perfect set, } \alpha < 2^\omega\}.$$
Let $\mathcal{F}$ be a large partition of $M$ of cardinality continuum into $NR$-sets.
By Lemma 3 there are perfect sets $[a, A] \subseteq [\omega]^\omega$ such that 
$$\mathcal{F}_{[a, A]} = \{F \cap [a, A]  \colon F \in \mathcal{F}\}$$
has cardinality continuum.

Thus for each $[a_\alpha, A_\alpha] \in \mathcal{P}$ we can choose disjoint sets  
$$B^{0}_{\alpha}, B^{1}_{\alpha} \in \{F \in \mathcal{F} \colon F \cap [a_\alpha, A_\alpha] \not = \emptyset\} \setminus (\{B^{0}_{\beta} \colon \beta < \alpha\}\cup \{B^{1}_{\beta} \colon \beta < \alpha\}).$$
For each $\varepsilon \in \{0, 1\}$ consider families 
$\mathcal{B}^{\varepsilon} = \{B^{\varepsilon}_{\alpha} \colon \alpha < 2^\omega\}$. Obviously by the construction 
$\mathcal{B}^{0}\cap \mathcal{B}^{1} = \emptyset.$ 
The sets $\bigcup \mathcal{B}^{\varepsilon}$ are not $(s)$-sets.

Indeed. Suppose that  $\bigcup \mathcal{B}^{\varepsilon}$ is an $(s)$-set for some $\varepsilon \in \{0, 1\}$. Then there exists $[a_\alpha, A_\alpha] \in \mathcal{P}$ such that $ [a_\alpha, A_\alpha]\cap \bigcup \mathcal{B}^{\varepsilon} = \emptyset$. But by the construction we have that 
$\{F \in \mathcal{F} \colon F \cap [a_\alpha, A_\alpha] \cap \bigcup \mathcal{B}^{\varepsilon} \not = \emptyset\}$ is non-empty. A contradiction.

If $\bigcup \mathcal{B}^{\varepsilon}$ is not an $(s)$-set for some $\varepsilon \in \{0, 1\}$, then there exists $[a_\alpha, A_\alpha] \in \mathcal{P}$ such that $[a_\alpha, A_\alpha] \subseteq \bigcup \mathcal{B}^{\varepsilon}$ and by the construction $[a_\alpha, A_\alpha] \cap \bigcup \mathcal{B}^{1 - \varepsilon} \not = \emptyset$ which contradicts with disjointness of families $\mathcal{B}^{0}$ and  $\mathcal{B}^{1}$
\end{proof}

\begin{theorem}
No large $EL$-dense and $EL$-open set $M \subseteq [\omega]^\omega$ admit a Kuratowski partition.
\end{theorem}

\begin{proof}
Let $\mathcal{F}$ be a large partition of $M$ of cardinality continuum into $NR$-sets. Suppose in contrary that $\mathcal{F}$ is a Kuratowski partition. 
By Lemm~3 for each  $[a, A] \subseteq [\omega]^\omega_{EL}$ such that 
$\mathcal{F}_{[a, A]} = \{F\cap [a, A] \colon F \in \mathcal{F}\}$
has cardinality continuum.
By Lemma 4 there exists $\mathcal{F}' \subseteq \mathcal{F}$ such that $\bigcup \mathcal{F}'$ is not a $CR$-set.
\end{proof}

{\sc Ryszard Frankiewicz}
\\
Institute of Mathematics, Polish Academy of Sciences, Warsaw, Poland,
\\
{\sl e-mail: rf@impan.pl}
\\

{\sc Joanna Jureczko}
\\
Wroc\l{}aw University of Science and Technology, Wroc\l{}aw, Poland,
\\
{\sl e-mail: joanna.jureczko@pwr.edu.pl}

\end{document}